\documentclass[12pt]{article}
\usepackage{epsfig}
\usepackage{amsfonts}
\usepackage{amssymb}
\usepackage{amstext}
\usepackage{amsmath}
\usepackage{xspace}
\usepackage{theorem}
\usepackage{graphicx}
\usepackage{array}
\usepackage{subfig}
\makeatletter
\newenvironment{proof}{{\bf Proof:  }}{\hfill\rule{2mm}{2mm}}

\newcommand{\junk}[1]{}
\newtheorem{theorem}{Theorem}
\newtheorem{lemma}[theorem]{Lemma}
\newtheorem{conjecture}{Conjecture}

\newtheorem{definition}{Definition}
\newtheorem{remark}{Remark}

\title{On the Effect of Data Dimensionality on Eigenvector Centrality}
\author{Gregory J. Clark, Felipe Thomaz, and Andrew Stephen\\
\small Sa\"id Business School\\[-0.8ex]
\small University of Oxford\\
\small \texttt{gregory.clark@sbs.ox.ac.uk  }\\
}
\begin{document}

\maketitle

\begin{abstract}
    Graphs (i.e., networks) have become an integral tool for the representation and analysis of relational data.  Advances in data gathering have lead to multi-relational data sets which exhibit greater depth and scope.  In certain cases, this data can be modeled using a hypergraph.  However, in practice analysts typically reduce the dimensionality of the data (whether consciously or otherwise) to accommodate a traditional graph model.  In recent years spectral hypergraph theory has emerged to study the eigenpairs of the adjacency hypermatrix of a uniform hypergraph.  We show how analyzing multi-relational data, via a hypermatrix associated to the aforementioned hypergraph, can lead to conclusions different from those when the data is projected down to its co-occurrence matrix.  In particular, we provide an example of a uniform hypergraph where the most central vertex (\`a la eigencentrality) changes depending on the order of the associated matrix.  To the best of our knowledge this is the first known hypergraph to exhibit this property.
\end{abstract}

\section{Introduction}

 There is a class of problems which seeks to quantify the importance of vertices (i.e., nodes) in a graph (i.e., network) according to some criterion.  %For example, measuring the influence \cite{LiIdentifying, Zhong, Zhong2, Calvo, Bucur, Mavrodiev, Iranzo,LiuLocating}, competitive advantage \cite{Bonacich1, Bonacich2, Podolny,BurtBook1, BurtBook2} and reputation \cite{Fraiberger} of actors in a social network.  
Centrality measures are typically employed in such cases.  Examples of such measures include degree, betweeness, closeness, and eigenvector centrality. While the aforementioned notions of centrality are related they can vary in practice.  This is particularly troublesome when two centrality measures identify different vertices as being `the most central'.  Famously, the Krakhardt kite is an example of a graph where different vertices have the greatest degree, betweeness, and closeness centrality \cite{Krackhardt}.  In a similar vein we construct a hypergraph whose most central vertex (by eigenvector centrality) changes depending on the order of the associated matrix (see Figure \ref{F:B_8}).  We explore this phenomenon by examining its underlying spectral properties.  

%The use of spectral properties to study empirical networks cannot be overstated.

%notion of eigencentrality `in two-dimensions' is analytically different from eigencentrality `in higher dimensions'.

In recent years, the principal eigenvector of the (normalized) adjacency hypermatrix of a $k$-uniform hypergraph (see \cite{Qi,Lim,Coo}) has received increasing attention as a way to model multi-relational data which faithfully analogizes the graph case \cite{Hu, ZhouAlt, Chen2017, Li2019, Fan, Cla0}.  Despite these developments, the term `hypergraph' has been historically employed in various contexts. For example, in \cite{ZhouAlt} the authors define the adjacency \emph{matrix} of a hypergraph to be $A = HWH^T-D$ where $H$ is the $|V|\times |E|$ incidence matrix, $W$ is a square matrix of edge-weights, and $D$ is the diagonal degree matrix.  Note that $A$ is precisely the co-occurrence matrix of $H$ when $W$ is the identity matrix.  This approach of using the co-occurrence matrix, or some variation thereof, as a stand-in for the adjacency hypermatrix has been the basis for a longstanding corpus of work.
To facilitate the adoption of spectral hypergraph theory in practice the field needs to overcome this historical momentum.  In particular spectral hypergraph theory needs to establish its computational feasibility and analytical novelty compared to traditional graph methods.  

There have been great strides in the computation of the principal eigenvector of a hypergraph as a constrained optimization problem \cite{Cha, QiBook}.  One can also consider the problem from a algebraic approach via the Lu-Man Method which was introduced in \cite{Lu} and further developed in \cite{Bai2, Zha}.  Herein we consider the question of novelty.  That is, how `different' is the principal eigenvector of a hypergraph from the graph formed from its co-occurrence matrix? We take a practical approach by addressing the following questions.  To what extent can the spectral ranking of a hypergraph and its co-occurrence matrix differ? Moreover, how much can these vectors vary coordinate-wise? 

We begin by presenting the necessary background for our discussion in the following section.  In Section \ref{S:Partite} we describe a property of the principal eigenvector of a $k$-partite $k$-uniform hypergraph. We leverage this property to construct a hypergraph whose most central vertex depends on the order of the accompanying hypermatrix in Section \ref{S:Bowtie}.  In Section \ref{S:Octa} we pivot to a structural approach and show how the loss of information incurred from depreciating data can lead to variations in the spectral ranking.  Finally, in Section \ref{S:Stars} we answer the second question by providing an upper bound on the Chebyshev distance between the aforementioned vectors.  We further provide a family of hypergraphs which achieves this bound in the limit.

\section{Preliminaries} 

% A \emph{multigraph} $G$ is an ordered pair $G = ([n],E)$ where $E \subseteq \binom{[n]}{2}$ is allowed to be a multiset.  We denote $\mu_G(e)$ to be the multiplicty of $e \in E$ and suppress the subscript when the context is clear.  The \emph{adjacency matrix} of a multigraph $A(G)$ is the $n \times n$ matrix where $A(G)_{i,j} = \mu(ij)$ if $ij \in E(G)$ and is zero otherwise.  From the Perron-Frobenius theorem we have that $A(G)$ has a unique strictly positive eigenpair $(\lambda, x)$ such that $||x||_2 = 1$ when $G$ is connected.  We refer to this eigenpair as the \textit{principal eigenpair} of $G$.  We now present the definition for the adjacency hypermatrix of a $k$-graph and maintain much of the notation of \cite{Coo}.

 A \emph{$k$-uniform hypergraph}, abbreviated \emph{$k$-graph}, is an ordered pair $H = ([n],E)$ where $E \subseteq \binom{[n]}{k}$.  Throughout we will assume that all hypergraphs are uniform and we reserve the language of ``hypergraph'' specifically for $k$-graphs where $k > 2$.  We maintain the notation of \cite{Coo}.  The \emph{(normalized) adjacency hypermatrix} of a $k$-graph $H$ is an \emph{order $k$ and dimension $n$ hypermatrix}, denoted ${\cal A}(H)$, which is a collection of $n^k$ elements where,
\begin{displaymath}
   a_{i_1,i_2, \dots, i_k} = \frac{1}{(k-1)!} \left\{
     \begin{array}{ll}
       1 & : \{i_1, i_2, \dots, i_k\} \in E(H)\\
       0 & : \text{otherwise}.
     \end{array}
   \right.
\end{displaymath}

 Let $H$ be a simple $k$-uniform hypergraph and $x \in \mathbb{C}^{|V|}$.  For $e \in E$ we denote  $x^e = \prod_{v \in e}x_v$.  The hypergraph $H$ defines a polynomial form,
\begin{equation}
\label{D:Max}
F_{{\cal{A}}(H)}(x) = k \cdot \sum_{e \in H} x^e.
\end{equation}

\begin{lemma}
\label{L:Princ}
(\cite{Coo}) In the case when $H$ is connected there is a unique strictly positive eigenpair $(\lambda, x)$ where $||x||_k = 1$ and
\[
\lambda = \max_{y : ||y||_k^k = 1}F_{{\cal A}(H)}(y).
\]
Moreover $x$ is the only strictly positive eigenvector which satisfies Equation \ref{D:Max}.
\end{lemma}

We refer to the eigenpair in Lemma \ref{L:Princ} as the \textit{principal eigenpair of $H$}.  More generally we have that $(\lambda, x)$ is an \emph{eigenpair} of $H$ if it satisfies the \emph{eigenequations}
\[
\lambda x_i^{k-1} = \sum_{
\substack{
e \in E \\
i \in e
}}
x^{e\setminus i} \text{ for }  i \in [n].
\]

The enterprise of this paper is to motivate the use of $k$-order hypergraphs to model $k$-relational data.  We do so by comparing the principal eigenvector of the normalized adjacency hypermatrix of $H$ with its co-occurrence matrix.
%As such, we have two rankings of the vertex set given by $y$ and $x$, respectively.  We provide an example of a hypergraph and its shadow which have different rank-1 vertices.  To our knowledge, this is the first known example of this phenomenon.  To motivate research in this direction we define the \textit{umbral index} to be the least $i$ for which the a hypergraph and its shadow have different vertices of rank $i$.

\begin{definition}
\label{D:shadow}
For a hypergraph $H = ([n], E)$, we define the \emph{clique-shadow} of $H$ as the multigraph
\[
\partial^* H = ([n], \{uv: uv \in e \in E(H)\})
\]
where $\mu(uv) = |\{e \in E(H) : u,v \in e\}|$ is the number of edges in $H$ containing $u$ and $v$.
\end{definition}

An example of a $3$-graph and its clique-shadow is given in Figure \ref{F:OnePleat}.  The term `clique-expansion' or `$2$-shadow' is sometimes used for Definition \ref{D:shadow}.  The nomenclature of `clique-shadow' was chosen instead to synthesize the language.  We also note that the \textit{shadow} of a hypergraph, denoted $\partial H$, replaces each $k$-edge of a hypergraph with all possible $(k-1)$-subedges.  To avoid confusion we adopt the notation of $\partial^*$.  

\begin{remark}
The adjacency matrix of $\partial^*(H)$ is the co-occurrence matrix of $H$.  Note that the co-occurrence matrix of $H$ preserves multiplicity.
\end{remark}

Throughout we consider a hypergraph $H$ and its clique-shadow $\partial^* H$.  For clarity, we reserve $(\rho, y)$ and $(\lambda, x)$ for the principal eigenpair of a hypergraph and a graph, respectively. 

We adopt the language of \textit{spectral ranking} as in \cite{Vigna} so that `the most central vertices by eigenvector centrality' have spectral rank 1.  We say that a hypergraph is \emph{opaque} if its spectral ranking differs from that of its clique-shadow.  We further define the \textit{umbral index} of a hypergraph $\mathfrak{u}(H)$ to be the least index for which the spectral ranking of $H$ and $\partial^*(H)$ differ.  In the case when $H$ is not opaque (i.e., $H$ and $\partial^*(H)$ have the same spectral ranking) we write $\mathfrak{u}(H) = 0$.

\section{Principle Eigenvector of a $k$-partite $k$-graph}
\label{S:Partite}

The following elegant result, given in \cite{Cioaba2}, provides insight into how vertices in an independent set `compete' for centrality in a bipartite graph.

\begin{theorem}
    \label{T:Sebi}
    (\cite{Cioaba2}) If $S$ is an independent set of a connected graph $G$ and $x$ is the principal eigenvector of $G$, then 
    \[
    \sum_{i \in S} x_i^2 \leq \frac{1}{2}.
    \]
    Equality happens if and only if $G$ is bipartite having $S$ as one color class.
\end{theorem}

We extend Theorem \ref{T:Sebi} to hypergraphs.  A $k$-graph is \emph{$k$-partite},
or a \emph{$k$-cylinder}, if its vertices can be partitioned into $k$ sets so that every edge uses exactly one vertex from each set \cite{Coo}. Our proof is similar to that of Theorem \ref{T:Sebi}.  We include it as it succinctly highlights the mechanisms underpinning this phenomenon. 

\begin{theorem}
\label{T:k-Sebi}
    If $S$ is an independent set of a connected $k$-uniform hypergraph $H$ and $y$ is the principal eigenvector of $H$, then
     \[
    \sum_{i \in S} y_i^k \leq \frac{1}{k}.
    \]   
    Equality occurs if and only if $H$ is a $k$-cylinder having $S$ as one color class.
\end{theorem}

\begin{proof}
For each $i \in S$ we have 
\[
\rho y_i^{k-1} =  \sum_{
\substack{
e \in E \\
i \in e}}y^{e\setminus{i}}.  
\]

Multiplying each equation by $y_i$ and summing over $i \in S$ yields
\[
\rho \sum_{i \in S}y_i^k = \sum_{i \in S}\sum_{i \in e} y^{e\setminus{i}}.
\]
Since $S$ is independent and the entries of $y$ are positive we have that
\[
\sum_{i \in S}\sum_{i \in e} y^{e\setminus{i}} \leq \frac{F_H(y)}{k} = \frac{\rho}{k}.
\]
The desired inequality follows whence $\rho > 0$ .

Moreover, equality occurs if and only if every edge $e$ which does not have a vertex in $S$ has $y_e = 0$.  Since $y$ is strictly positive this implies that every edge must have a vertex in $S$.  Whence $S$ is an independent set it follows that each edge of $H$ contains exactly one node from $S$.  Thus $S$ is a color class of $H$. 
\end{proof}

Theorem \ref{T:k-Sebi} shows that vertices in the same color class in a $k$-partite $k$-graph `compete for centrality' like in a zero-sum game.  However, this relationship breaks down when the number of color classes exceeds the order of the (hyper)graph. Note that the clique-shadow operation preserves the color classes of a hypergraph so that vertices in a $k$-partite $k$-graph `compete' in a stricter sense than in the clique-shadow.  This observation forms the basis of our constructions moving forward.

%We prove Theorem \ref{T:Cheby} by applying Theorem \ref{T:k-Sebi} to the family of $k$-stars.  Consider now the $k$-order star graph ${\cal{S}}(\eta,k-1)$ which consists of $\eta$ $k$-edges all sharing a common vertex. That is, 
%\[
%E({\cal S}(\eta,k-1)) = \{[1,(i-1)(k-1)+2, \dots, (i-1)(k-1)+k] : 1 \leq i \leq \eta\}.
%\]
%The clique-shadow of a $k$-star is a windmill graph (aka $n$-fan, friendship graph).  In \cite{Estrada}, the author determined the spectrum and network properties of windmill graphs.  Adhering to their notation, the windmill graph $W(\eta, k-1)$ consists of $\eta$ copies of the complete graph $K_{k}$ joined at a single vertex.  We provide a drawing of ${\cal{S}}(3,2)$ and $W(3,2)$ in Figure \ref{F:Windmill}.    With this notation we have $\partial^*({\cal S}(\eta,k-1)) = W(\eta,k-1)$.  

\section{Generalized Bowtie}
\label{S:Bowtie}

We construct a hypergraph which has the property that it and its clique-shadow identify different vertices as being the most central under eigenvector centrality.   Consider the following  \emph{$t$-pleated bowtie} $3$-graph,
\[
B_t = \{[c,\ell_1,\ell_2],[c,r_1,r_2]\} \cup \left(\bigcup_{i =1}^t \{[c,\ell_1, \ell_{i+1}], [r_1, r_2, r_{i+2}]\} \right).
\]

A drawing of $B_1$ and $\partial^* B_1$ is given in Figure \ref{F:OnePleat}. 

We establish a spectral characterization for the spectral rank 1 vertices of $B_t$ and  $\partial^*B_t$, respectively. In particular we show that $r_1$ and $r_2$ are the most central vertices in $B_t$ precisely when $\rho$ is large (compared to $t$).  We further show that vertex $c$ is the most central in $\partial^*B_t$ precisely when $\lambda$ is big (compared to $t$).  We conclude by showing that both $\rho$ and $\lambda$ are sufficiently large when $t = 8$. A drawing of $B_8$ is given in Figure \ref{F:B_8}. 

\begin{figure}[ht]
    \centering
    \includegraphics[width=0.85\textwidth]{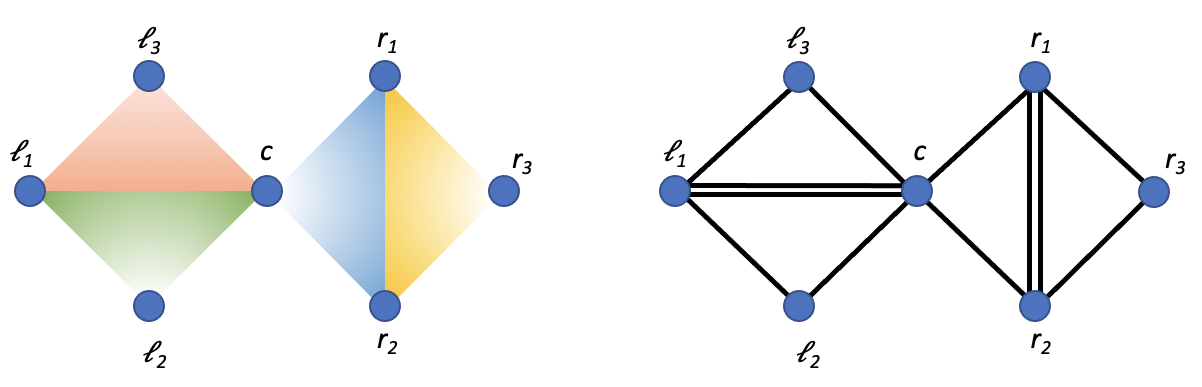}
        \caption{The one-pleated bowtie, $B_1$ where edges are drawn as triangular faces (left), and its clique-shadow (right).}
    \label{F:OnePleat}
\end{figure}

\begin{lemma}
\label{L:Bound}
Fix $t$ and let $(\rho, y)$ be the principal eigenvector of $B_t$.  Then the spectral rank 1 vertices of $B_t$ are $\{r_1, r_2\}$ if and only if $t < \rho \sqrt{\rho-1}-1$.
\end{lemma}

\begin{proof}
For simplicity we write 
\[
\alpha = y(c), \beta = y(r_1) = y(r_2), \gamma = y(r_i) \text{ for }i > 2, \delta = y(\ell_1)\text{ and }\varepsilon = y(\ell_j) \text{ for }j > 1.
\]
As such, the eigenequations of $ B_t$ are

 \begin{displaymath}
\left\{
     \begin{array}{ll}
    \rho \alpha^2 &= \beta^2 + (t+1)\delta\varepsilon \\
    \rho \beta^2 &= \alpha\beta  + t\beta\gamma \\
    \rho \gamma^2 &= \beta^2 \\
    \rho \delta^2 &= (t+1)\alpha\varepsilon \\
    \rho \varepsilon^2 &= \alpha\delta
     \end{array}
   \right.
\end{displaymath}
We show that $\beta > \alpha > \gamma, \delta, \varepsilon$ if and only if $t < \rho\sqrt{\rho -1}$.
From the last two eigenequations we have
\[
\frac{\rho \delta^2}{(t+1)\varepsilon} = \alpha = \frac{\rho \varepsilon^2}{\delta}.
\]
It follows that $\delta^3 = (t+1)\varepsilon^3$.  We further have
\begin{align*}
    \rho \alpha^2 &= \beta^2 + (t+1)\delta\varepsilon = \beta^2 + (t+1)^{4/3}\varepsilon^2 \\
    &=\beta^2 +(t+1)^{4/3}\left(\frac{\alpha^2(t+1)^{2/3}}{\rho^2}\right) =\beta^2 +  \frac{\alpha^2(t+1)^2}{\rho^2}.
\end{align*}
Thus
\[
\left(\frac{\alpha}{\beta}\right)^2 =\frac{\rho^2}{\rho^3-(t+1)^2}.
\]
Hence $\beta > \alpha$ if and only if $t < \rho\sqrt{\rho - 1} - 1$.  %An exact solution can be found using the general cubic formula.

The difference of first and third eigenequation yield $\alpha > \gamma$ .  We now show $\alpha > \delta$.  Rearranging the fourth eigenequation yields 
\[
\varepsilon = \frac{\rho \delta^2}{(t+1)\alpha}.
\]
Substitution into fifth eigenequation yields
\[
\left(\frac{\alpha}{\delta}\right)^3 = \frac{\rho^3}{(t+1)^2}.
\]
Indeed, $\alpha > \delta$ if and only if $\rho^3 > (t+1)^2$.  This inequality is satisfied when $\rho\sqrt{\rho - 1}-1 > t$.  Finally we take the ratio of fourth eigenequation multiplied by $\delta$ and the fifth eigenequation multiplied by $\varepsilon$ to conclude
\[
\left(\frac{\delta}{\varepsilon}\right)^3 = t+1
\]
so that $\delta/\varepsilon > \sqrt[3]{t+1}$ implying that $\delta > \varepsilon$.  It follows that $\alpha > \varepsilon$ concluding the proof.
\end{proof}

We now establish a similar characterization for $\partial^* B_t$.  

\begin{lemma}
\label{L:ShadowBound}
Fix $t$ and let $(\lambda, x)$ be the principal eigenvector of $\partial^* B_t$.  We have that the only vertex of spectral rank 1 is $c$ if and only if $t < (2\lambda - \lambda^2)/(\lambda+1)$ or equivalently 
\[
\lambda > \frac{t+2+\sqrt{t^2 + 12t + 4}}{2}.
\]

\end{lemma}

\begin{proof}
For simplicity, we write 
\[
a = x(c), b = x(r_1) = x(r_2), c = x(r_i) \text{ for }i > 2, d = x(\ell_1), \text{ and }e = x(\ell_j) \text{ for }j > 1.
\]
As such, the eigenequations of $\partial^* B_t$ are

 \begin{displaymath}
\left\{
     \begin{array}{ll}
    \lambda a &= 2b + (t+1)(d+e) \\
    \lambda b &= a+(t+1)b +tc \\
    \lambda c &= 2b \\
    \lambda d &= (t+1)(a+e) \\
    \lambda e &= a + d
     \end{array}
   \right.
\end{displaymath}

We prove our claim by showing that $a > b,c,d,e$ when $\lambda$ is sufficiently large.

Substituting the third eigenequation into the second yields
\[
    \lambda b = a + (t+1)b + tc = a + (t+1)b + t(2b/\lambda)
\]
so that
\[
    \frac{a}{b} = \frac{\lambda^2 - \lambda(t+1) - 2t}{\lambda}.
\]
Indeed $a > b$ if and only if $t < (\lambda^2-2\lambda)/(\lambda+2)$ or equivalently 
\[
\lambda > \frac{t+2+\sqrt{t^2 + 12t + 4}}{2}.
\]

Taking the difference between the first and third eigenequations yields
\[
\lambda(a-c) = (t+1)(d+e).
\]
Whence $\lambda, t, d, e > 0$ we have that $a > c$.

Now consider the difference of the first and fourth eigenequations,
\[
\lambda(a-d) = 2b + (t+1)(d-a).
\]
Suppose to the contrary that $a \leq d$.  Then $\lambda(a-d) \leq 0$. Moreover, we have that $2b + (t+1)(d-a) > 0$ since $b, t >0$ and $d-a \geq 0$. This implies that $0 < 0$ which is a contradiction.  It must be the case that $a > d$.  

We conclude by showing $d > e$.  Substituting the fifth eigenequation into the fourth yields
\[
\lambda d = (t+1)(\lambda e - d + e)
\]
so that
\[
d = e \left(\frac{(t+1)(\lambda+1)}{\lambda - (t+1)}\right).
\]
Since 
\[
(t+1)(\lambda + 1) > \lambda - (t+1)
\]
we have that $d > e$ as desired. 
\end{proof}

\begin{figure}[ht]
    \centering
    \includegraphics[width=0.55\textwidth]{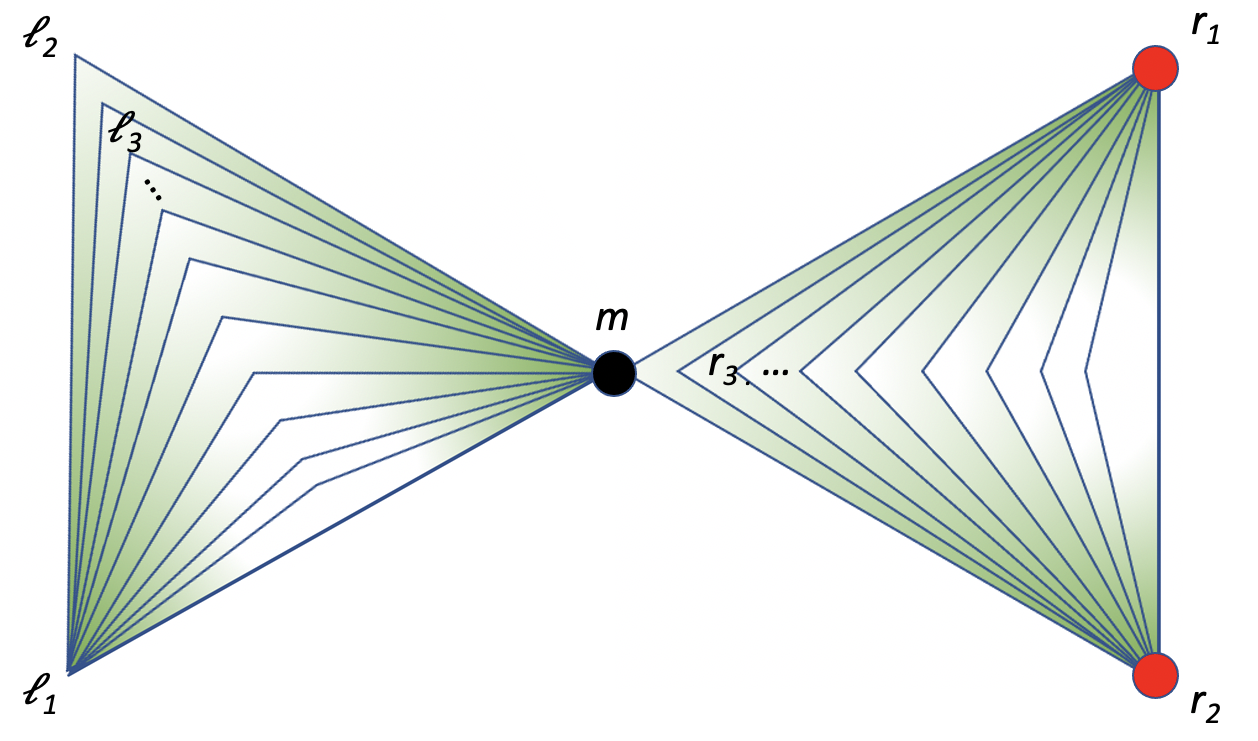}
    \caption{The 8-pleated bowtie, $B_8$.}
    \label{F:B_8}
\end{figure}

We now prove our main result.

\begin{theorem}
\label{T:B8}
$\mathfrak{u}(B_8) = 1$.  That is, the spectral rank 1 vertices of $B_8$ and its clique-shadow $\partial^*B_8$ are distinct. 
\end{theorem}

\begin{proof}
We adhere to the notation of Lemmas \ref{L:Bound} and \ref{L:ShadowBound}.  From Theorem \ref{T:k-Sebi} we have that $0 < \alpha, \beta < (1/3)^{1/3}$ and 
\begin{align*}
   1/3 &= \alpha^3 + t\gamma^3 \\
   1/3 &= \beta^3 + (t+1)\varepsilon^3 \\
   1/3 &= \beta^3 + \delta^3.
\end{align*}
It follows that
\begin{align*}
    \gamma &= \left(\frac{1/3-\alpha^3}{t}\right)^{1/3} \\
    \delta &= (1/3 - \beta^3)^{1/3} \\
    \varepsilon &= \left(\frac{1/3 - \beta^3}{t+1}\right)^{1/3}.
\end{align*}

Appealing to Equation \ref{D:Max} we find that the polynomial form $F_{B_t}$ is a function only of $t, \alpha,$ and $\beta$. Consider $y'$ where $\alpha = \beta = (1/3)^{1/3} - 1/8$. We have then that 
\[
\rho \geq F_{B_8}(y') \geq 4.68949
\]
so that $\rho\sqrt{\rho-1}-1 > 8.007 > t$.

Further we find that the minimal polynomial of $\lambda$ is 
\[
m_\lambda(x) = x^{5} - 9 \, x^{4} - 98 \, x^{3} + 592 \, x^{2} + 2448 \, x + 2048
\]
so that $\lambda > 11.097$ and thus 
\[
\lambda > \frac{t+\sqrt{t^2+12t+4}}{2} \approx 10.403.
\]
The conclusion follows immediately from Lemmas \ref{L:Bound} and \ref{L:ShadowBound}.
\end{proof}

\begin{conjecture}
We have $\mathfrak{u}(B_t) = 1$ for $t \geq 8$.
\end{conjecture}

\section{Modified Octahedron}
\label{S:Octa}

Consider the modified octahedron given in Figure \ref{F:Octa}.  Let $O_R$ denote the 3-graph formed by taking the red faces of the octahedron and the green edge.  To be precise,
\[
E(O_R) = \{[t,p,q],[t,r,s],[b,q,r],[b,p,s],[u,p,q]\}.
\]
Similarly define $O_B$ to be the 3-graph formed from the blue faces of the octahedron and the green edge,
\[
E(O_B) = \{[t,p,r],[t,p,s],[b,p,q],[b,r,s],[u,p,q]\}.
\]
An approximation of the principal eigenvectors of $O_R, O_B, \partial^* O_R$, and $\partial^* O_B$ is given in Table \ref{T:Octahedron_eigen} and the corresponding spectral rankings are given in Table \ref{T:Octahedron}.
\begin{definition}
We say that two hypergraphs $H_1$ and $H_2$ are co-umbral mates if $H_1 \neq H_2$ but $\partial^*(H_1) = \partial^*(H_2)$.
\end{definition}

Co-umbral mates demonstrate the loss of information incurred by using the co-occurrence matrix of a hypergraph instead of its adjacency hypermatrix.  Observe that $O_R$ and $O_B$ are co-umbral mates as $O_R \neq O_B$ and $\partial^* O_R = \partial^* O_B$.  Moreover, $t$ are $b$ are identically situated in the clique-shadow while occupying distinct positions in $O_R$ (and $O_B$).  We thus expect the spectral ranking of the red/blue modified octahedron and its clique-shadow to differ.

\begin{figure}%
    \centering
    \subfloat[\centering A Modified Octahedron]{{\includegraphics[width=4cm]{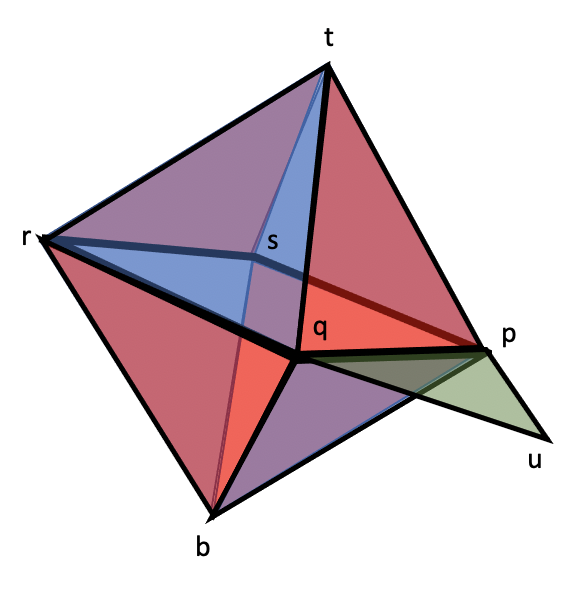} }}%
    \qquad
    \subfloat[\centering $O_R$]{{\includegraphics[width=4.2cm]{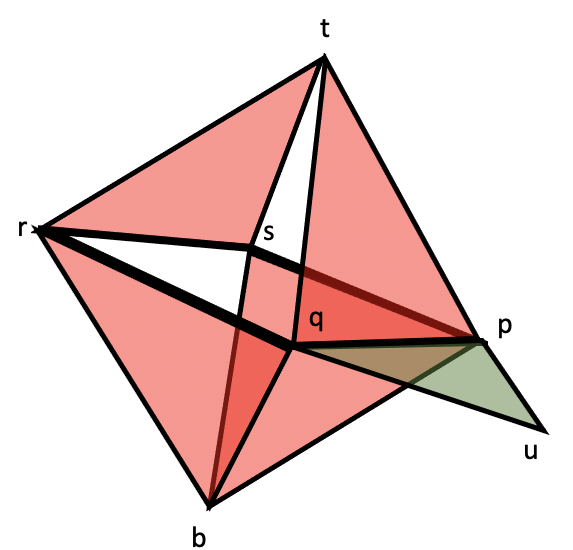} }}%
    \qquad
    \subfloat[\centering $O_B$]{{\includegraphics[width=4cm]{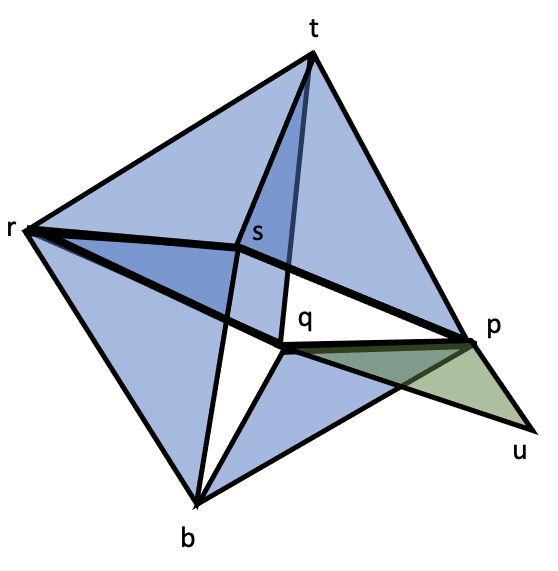} }}%
    \caption{A modified octahedron, $O_R$, and $B_R$, respectively.}%
    \label{F:Octa}%
\end{figure}

%\begin{lemma}
%We have $\mathfrak{u}(O_R) =\mathfrak{u}(O_B) = 2 $.
%\end{lemma}

%\begin{proof}
%We prove our claim via Lemma \ref{L:Princ}.  Using \emph{minimized\_constrained} (via Sage \cite{Sag}) we approximate $O_R$ to find that (rounding to five digits),
%\[
%x_ p = x_q = 0.5938, x_t = 0.5159,  x_b = 0.5121 , x_r = x_s = 0.4986, x_m = 0.3951.
%\]

%We further have from \emph{eigenvectors\_right}
%\[
%y_p = y_q = 0.4871, y_t = y_b = 0.3579, y_r = y_s = 0.3348, y_m = 0.2121.
%\]
%A summary is provided in Table \ref{T:Octahedron_eigen}.

%Indeed ${\mathfrak{u}}(O_R) = 2$.  
%\end{proof}

\begin{theorem}
$O_R$ and $O_B$ are opaque.
\end{theorem}

\begin{proof}
Consider $O_R$. Let $(\rho, y)$ and $(\lambda, x)$ be the principal eigenvector of $O_R$ and $\partial^*(O_R)$, respectively. We remark that $x_t = x_b$ by symmetry of $\partial^*(O_R)$.  It remains to be shown that $y_t \neq y_b$.  We will abuse notation and write $v$ for $y_v$, the value of the principal eigenvector of the vertex $v$.  By symmetry of $O_R$ we have $p = q$ and $r = s$.  

First suppose to the contrary that $q=r$.  Taking the difference of their eigenequations yields
\[
    \rho q^2 - \rho r^2 = (rb + tq + uq) - (tr +qb)
\]
so that $ 0 = uq$.  It follows that $u = 0$ or $q = 0$.  This cannot be the case as the principal eigenvector is non-negative.  Indeed $r \neq q$.

Now consider the difference of the eigenequations of $t$ and $b$,
\[
\rho t^2 - \rho b^2 = rs+ pq - (rq + ps)  = r^2 + q^2 - 2rq = (r-q)^2.
\]
Since $r \neq q$ we have that $(r-q)^2 > 0$ which implies that $t > b$.

\end{proof}

\begin{table}
\caption{An approximation of the principal eigenvector of $O_R, O_B, \partial^* O_R$ and $\partial^* O_B$.  Approximations were computed via Sage \cite{Sag}.}
\begin{center}
\begin{tabular}{ |c|c|c|c| } 
 \hline
 Vertex & $y(O_R)$ & $y(O_B)$ & $x(\partial^*(O_R))=x(\partial^*(O_B))$ \\  \hline
 $p$ & 0.5938 & 0.5938 & 0.4871 \\ 
 $q$ &  0.5938 & 0.5938 & 0.4871 \\ 
 $t$ & 0.5159 & 0.5121 & 0.3579 \\ 
 $b$ & 0.5121 & 0.5159 &0.3579 \\
 $r$ & 0.4986 & 0.4986 &0.3348\\
 $s$ & 0.4986 & 0.4986 &0.3348\\
 $u$ & 0.3951 &0.3951 & 0.2121 \\ \hline
\end{tabular}
\label{T:Octahedron_eigen}
\end{center}
\end{table}

\begin{table}
\caption{The spectral ranking of $O_R, O_B, \partial^* O_R$ and $\partial^* O_B$.}
\begin{center}
\begin{tabular}{ |c|c|c|c| } 
 \hline
 Rank & $O_R$ & $O_B$ & $\partial^*(O_R) = \partial^*(O_R)$ \\  \hline
 1 & $p,q$ & $p,q$ & $p,q$ \\ 
 2&  $t$ & $b$ & $t,b$ \\ 
 3 & $b$ & $t$ & $r,s$ \\ 
 4 & $r,s$ & $r,s$ &$u$\\
 5 & $u$ & $u$ & $\emptyset$ \\\hline
\end{tabular}
\label{T:Octahedron}
\end{center}
\end{table}

\section{Hyperstars and Windmills}
\label{S:Stars}

Consider now the $k$-order star graph ${\cal{S}}(\eta,k)$ which consists of $\eta$ $k$-edges all sharing a common vertex. That is, 
\[
E({\cal S}(\eta,k)) = \{[1,(i-1)(k-1)+2, \dots, (i-1)(k-1)+k] : 1 \leq i \leq \eta\}.
\]
The clique-shadow of a $k$-star is a windmill graph (aka fan or friendship graph).  In \cite{Estrada}, the author determined the spectrum and network properties of windmill graphs.  Adhering to their notation, the windmill graph $W(\eta, k)$ consists of $\eta$ copies of the complete graph $K_{k}$ joined at a single vertex.  We provide a drawing of ${\cal{S}}(3,3)$ and $W(3,2)$ in Figure \ref{F:Windmill}.    With this notation we have $\partial^*({\cal S}(\eta,k)) = W(\eta,k-1)$.

\begin{figure}[ht]
    \centering
    \includegraphics[width=0.65\textwidth]{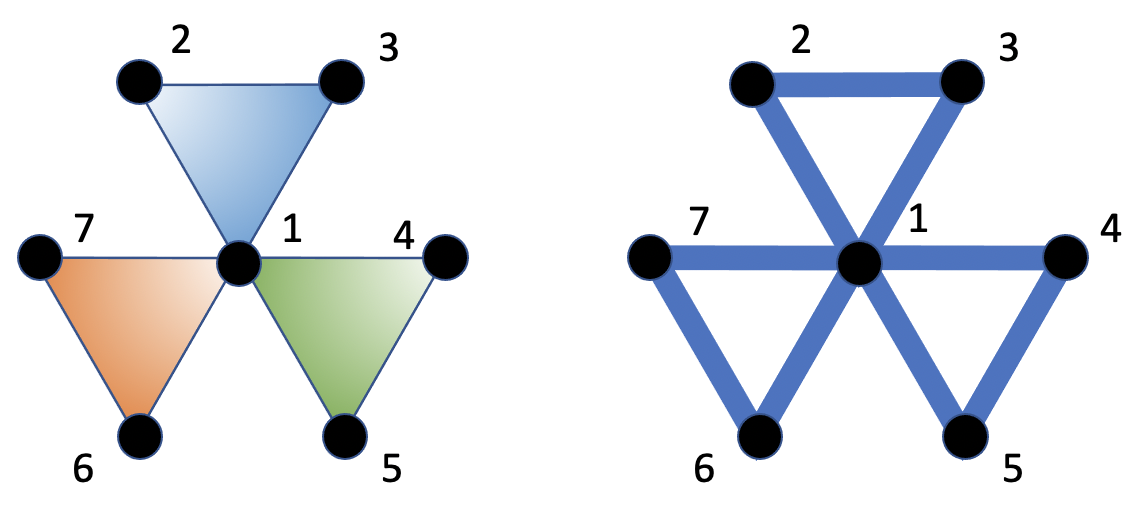}
    \caption{${\cal{S}}(3,3)$ and its shadow, $W(3,2)$, respectively.}
    \label{F:Windmill}
\end{figure}

\begin{theorem}
\label{T:Cheby}
Let $y$ and $x$ be the principal eigenvectors of $H$ and $\partial^*(H)$, respectively. Consider $\hat y := (y_v^k)_{v \in V}$.  Then the Chebyshev distance between 
\[
D(\hat y, \hat x) = \max_v | y_v^k - x_v^2| \leq 1/2.
\]
Moreover, for 
\[
\Delta_k := \max_{H \in {{\cal H}(k)}} D(\hat y, \hat x),
\]
where the maximum is taken over all connected $k$-graphs, we have 
\[
\lim_{k \to \infty} \Delta_k = 1/2.
\]
\end{theorem}

%**We first consider properties of the eigencentrality measure and then consider the spectral ranking. 
%***The umbral index of the star is 0 but its coordinate distance achieves the maximum. 

%For a hypergraph $H$ let $\delta(H) = \max |y^k - x^2|$ be the greatest distance between the eigencentrality measure of nodes in $H$ and its shadow. **Just the Chebyshev distance*** Let $\Delta_k = \max_{H \in \mathcal{H}_k} \delta_H$ where $\mathcal{H}_k$ is the set of all connected $k$-graphs.  We have by our theorem that $\Delta_k \leq 1/2$.  We show the family of windmill hypergraphs gives $\lim \Delta_k = 1/2$.  So that there exist hypergraphs whose maximum coordinate difference achieves the worst possible. 

%We prove Theorem \ref{T:Cheby} by leveraging a property of $k$-cylinders. A $k$-graph is \emph{$k$-partite},
%or a \emph{$k$-cylinder}, if its vertices can be partitioned into $k$ sets so that every edge uses exactly one vertex from each set \cite{Coo}. 

%We combine results from \cite{Estrada} to determine the weight of the central vertex in $W(\eta, k)$.

\begin{proof}
From Theorem \ref{T:k-Sebi} we have that $D(\hat y, \hat x)  \leq 1/2$.  Let $(\lambda, x)$ be the principal eigenpair of $W(\eta, k)$. From \cite{Estrada} we have 
\[
\lambda = \frac{k-1}{2} + \sqrt{\left(\frac{k-1}{2}\right)^2 + \eta k}
\]
where $\lambda x_1 = 1 - x_1$ and $x_i = (1-x_1)/(\eta k)$ for $i > 1$.  Solving for $x_1$ and normalizing such that $||x||_2 = 1$ yields

\[
x_1 = \frac{2}{\zeta \sqrt{\frac{\left(\frac{2}{\zeta}-1\right)^2}{k \eta} + \frac{4}{\zeta^2}}} \text{ for }
\zeta = k + \sqrt{(k-1)^2 + 4k \eta}+1.
\]

For fixed $k$, $\lim_{\eta \to \infty} \hat x_1(\eta) = 1/2$. Now consider ${\cal S}(n,k)$ which is a $k$-cylinder where vertex-$1$ forms a color class. Appealing to Theorem \ref{T:k-Sebi} we have that $\hat y_1 = 1/k$.  We have shown
 \[
 \Delta_k \geq \lim_{\eta \to \infty} |\hat y_1({\cal S}(\eta, k))- \hat x_1(W(\eta, k-1)| = 1/2-1/k.
 \]
Indeed $\lim_k \Delta_k = 1/2$ as desired. 
\end{proof}

\begin{conjecture}
$\Delta_k = 1/2 + o(1)$ for all $k \geq 3$.
\end{conjecture}

%We conclude with the following observation. In \cite{Estrada} it was shown that the windmill graph has the property that the average Watts-Strogatz clustering coefficient \cite{Watts} and transitivity diverge with $\eta$.  We observe a similar divergence with the Chebyshev distance between powers of the principal eigenvector of a graph and its clique-shadow.  This begs the question.
%\begin{question}
%Is the principal eigenvector of a hypergraph ($k > 2)$ related to its transitivity (a global property) while the principal eigenevector of a graph $(k = 2)$ is related to the average Watts-Strogatz clustering coefficient (a local property)?
%\end{question} 

\bibliography{main}{}

\begin{thebibliography}{10}

\bibitem{Krackhardt}
David Krackhardt and Robert~N Stern.
\newblock Informal networks and organizational crises: An experimental
  simulation.
\newblock {\em Social psychology quarterly}, pages 123--140, 1988.

\bibitem{Qi}
Liqun Qi.
\newblock Eigenvalues of a real supersymmetric tensor.
\newblock {\em J. Symbolic Comput.}, 40(6):1302--1324, 2005.

\bibitem{Lim}
L.~H. Lim.
\newblock {Singular values and eigenvalues of tensors: a variational approach}.
\newblock {\em Pro- ceedings of the IEEE International Workshop on
  Computational Advances in Multi-Sensor Adaptive Processing}, 1:129--132, May
  2005.

\bibitem{Coo}
Joshua Cooper and Aaron Dutle.
\newblock Spectra of uniform hypergraphs.
\newblock {\em Linear Algebra Appl.}, 436(9):3268--3292, 2012.

\bibitem{Hu}
Shenglong Hu and Ke~Ye.
\newblock Multiplicities of tensor eigenvalues.
\newblock {\em Commun. Math. Sci.}, 14(4):1049--1071, 2016.

\bibitem{ZhouAlt}
Dengyong Zhou, Jiayuan Huang, and Bernhard Sch{\"o}lkopf.
\newblock Learning with hypergraphs: Clustering, classification, and embedding.
\newblock {\em Advances in neural information processing systems},
  19:1601--1608, 2006.

\bibitem{Chen2017}
Yannan Chen, Liqun Qi, and Xiaoyan Zhang.
\newblock The fiedler vector of a laplacian tensor for hypergraph partitioning.
\newblock {\em SIAM Journal on Scientific Computing}, 39(6):A2508--A2537, 2017.

\bibitem{Li2019}
Hong-Hai Li and Bojan Mohar.
\newblock On the first and second eigenvalue of finite and infinite uniform
  hypergraphs.
\newblock {\em Proceedings of the American Mathematical Society},
  147(3):933--946, 2019.

\bibitem{Fan}
Y.-Z. {Fan}, T.~{Huang}, Y.-H. {Bao}, C.-L. {Zhuan-Sun}, and Y.-P. {Li}.
\newblock {The spectral symmetry of weakly irreducible nonnegative tensors and
  connected hypergraphs}.
\newblock {\em ArXiv e-prints}, April 2017.

\bibitem{Cla0}
Gregory~J. Clark and Joshua~N. Cooper.
\newblock A harary-sachs theorem for hypergraphs.
\newblock {\em Journal of Combinatorial Theory, Series B}, 149:1--15, 2021.

\bibitem{Cha}
Jingya Chang, Weiyang Ding, Liqun Qi, and Hong Yan.
\newblock Computing the p-spectral radii of uniform hypergraphs with
  applications.
\newblock {\em Journal of Scientific Computing}, 75(1):1--25, 4 2018.

\bibitem{QiBook}
Liqun Qi and Ziyan Luo.
\newblock {\em Tensor analysis: spectral theory and special tensors}, volume
  151.
\newblock Siam, 2017.

\bibitem{Lu}
Linyuan Lu and Shoudong Man.
\newblock Connected hypergraphs with small spectral radius.
\newblock {\em Linear Algebra Appl.}, 509:206--227, 2016.

\bibitem{Bai2}
Shuliang Bai and Linyuan Lu.
\newblock A bound on the spectral radius of hypergraphs with {$e$} edges.
\newblock {\em Linear Algebra Appl.}, 549:203--218, 2018.

\bibitem{Zha}
Wei Zhang, Liying Kang, Erfang Shan, and Yanqin Bai.
\newblock The spectra of uniform hypertrees.
\newblock {\em Linear Algebra Appl.}, 533:84--94, 2017.

\bibitem{Vigna}
Sebastiano Vigna.
\newblock Spectral ranking.
\newblock {\em Network Science}, 4(4):433--445, 2016.

\bibitem{Cioaba2}
Sebastian Cioab{\u{a}}.
\newblock A necessary and sufficient eigenvector condition for a connected
  graph to be bipartite.
\newblock {\em The Electronic Journal of Linear Algebra}, 20:351--353, 2010.

\bibitem{Sag}
The {S}age-{C}ombinat community.
\newblock {S}age-{C}ombinat: enhancing {S}age as a toolbox for computer
  exploration in algebraic combinatorics, 2008.
\newblock {\tt http://sagemath.org}.

\bibitem{Estrada}
Ernesto Estrada.
\newblock When local and global clustering of networks diverge.
\newblock {\em Linear Algebra and its Applications}, 488:249--263, 2016.

\end{thebibliography}
\bibliographystyle{unsrt}

\end{document}